\documentclass[12pt]{article}
\usepackage{amsmath,amsbsy,amscd,amssymb,graphicx,epsfig,color}
\graphicspath{{./figuras/}}
\usepackage{amssymb,amsmath, amsfonts, amsthm}
\usepackage{bbding}
\usepackage[latin1]{inputenc}
\usepackage{enumerate}

\usepackage[]{graphicx}
\usepackage[]{subfigure}
\usepackage{float}
\usepackage{amssymb}
\usepackage{amsfonts}
\usepackage{float}
\graphicspath{{./figures/}}

\newcommand{\norm}[1]{\left\Vert#1\right\Vert}

\newcommand{\abs}[1]{\left\vert#1\right\vert}
\newcommand{\Y}{\mathcal{Y}}
\newcommand{\X}{\mathcal{X}}
\newcommand{\D}{\mathcal{D}}

\newcommand{\ds}{\displaystyle}
\newcommand{\s}{\scriptscriptstyle}
\newcommand{\W}{W_{\s 0, \theta}}
\newcommand{\chig}{\scalebox{1.3}{$\chi$}}

 \newtheorem{thm}{Theorem}[section]
 
 \newtheorem{lem}[thm]{Lemma}
 
 \newtheorem{defn}[thm]{Definition}
 \newtheorem{rem}[thm]{Remark}

\begin{document}

\title{Mixed spatially varying $L^2$-BV regularization of inverse ill-posed problems 
}
\author{Gisela L. Mazzieri\thanks{Instituto de Matem\'{a}tica Aplicada del Litoral, IMAL,
CONICET-UNL, G\"{u}emes 3450, S3000GLN, Santa Fe, Argentina, Departamento de Matem\'{a}tica,
 Facultad de Bioqu\'{\i}mica y Ciencias Biol\'{o}gicas, Universidad Nacional del Litoral, Santa Fe, Argentina ({\tt glmazzieri@santafe-conicet.gov.ar}).}
\and Ruben D. Spies$^\text{\Envelope,\,}$\thanks{Instituto de
Matem\'{a}tica Aplicada del Litoral, IMAL, CONICET-UNL, G\"{u}emes 3450,
S3000GLN, Santa Fe, Argentina and Departamento de Matem\'{a}tica,
Facultad de Ingenier\'{\i}a Qu\'{\i}mica, Universidad Nacional del Litoral,
Santa Fe, Argentina (\Envelope\,: {\tt rspies@santafe-conicet.gov.ar}).}
\and Karina G. Temperini \thanks{Instituto de Matem\'{a}tica Aplicada del Litoral, IMAL,
CONICET-UNL, G\"{u}emes 3450, S3000GLN, Santa Fe, Argentina, Departamento de Matem\'{a}tica,
 Facultad de Humanidades y Ciencias , Universidad Nacional del Litoral, Santa Fe, Argentina ({\tt ktemperini@santafe-conicet.gov.ar}).}}

%
%
%

\maketitle
\begin{abstract} Several generalizations of the traditional Tikhonov-Phillips regularization method
have been proposed during the last two decades. Many of these generalizations are based upon inducing stability
throughout the use of different penalizers which allow the capturing of diverse properties of the exact solution
(e.g.\;edges, discontinuities, borders, etc.). However, in some problems in which it is known that the regularity of
the exact solution is heterogeneous and/or anisotropic, it is reasonable to think that a much better option could be
the simultaneous use of two or more penalizers of different nature. Such is the case, for instance, in some image
restoration problems in which preservation of edges, borders or discontinuities is an important matter. In this
work we present some results on the simultaneous use of penalizers of $L^2$ and of bounded variation (BV) type. For
particular cases, existence and uniqueness results are proved. Open problems are discussed and results to signal  restoration problems are presented.
\end{abstract}

{\thispagestyle{empty}} 

\section{Introduction and preliminaries}

For our general setting we consider the problem of finding $u$ in an equation of the form
\begin{equation}\label{eq:prob-inv}
Tu=v,
\end{equation}
where $T:\X\rightarrow \Y$ is a bounded linear operator between two infinite dimensional Hilbert spaces  $\X$ and
$\Y$, the range of $T$ is non-closed and $v$ is the data, which is supposed to be known, perhaps with a certain
degree of error. In the sequel and unless otherwise specified, the space $\X$ will be $L^2(\Omega)$ where $\Omega\subset\mathbb{R}^n$ is a bounded open convex set with Lipschitz boundary.
It is well known that under these hypotheses problem (\ref{eq:prob-inv}) is ill-posed in the sense of Hadamard (\cite{ref:Hadamard-1902}) and it must
be regularized before any attempt to approximate its solutions is made (\cite{refb:Engl-Hanke-96}). The most usual way of
regularizing a problem is by means of the use of the \textit{Tikhonov-Phillips regularization method} whose general
formulation can be given within the context of an unconstrained optimization problem. In fact, given an appropriate
penalizer $W(u)$ with domain $\mathcal D\subset \X$, the regularized solution obtained by the Tikhonov-Phillips
method and such a penalizer, is the minimizer $u_\alpha$, over $\mathcal D$, of the functional
\begin{equation}\label{eq:formulacion W}
J_{\alpha,W}(u)= \norm{Tu-v}^2+\alpha W(u),
\end{equation}
where $\alpha$ is a positive constant called regularization parameter. For general penalizers $W$, sufficient
conditions guaranteeing existence, uniqueness and weak and strong stability of the minimizers under different types
of perturbations where found in \cite{ref:Mazzieri-Spies-Temperini-JMAA-2012}.

Each choice of an appropriate penalizer $W$ originates a different regularization method producing a particular
regularized solution possessing particular properties. Thus, for instance, the choice of $W(u)=\norm{u}^2_{\scriptscriptstyle L^2(\Omega)}$ gives
raise to the classical Tikhonov-Phillips method of order zero  producing always smooth regularized approximations which approximate, as $\alpha\to 0^+$, the best approximate solution (i.e. the least squares solution of minimum norm) of problem (\ref{eq:prob-inv}) (see \cite{refb:Engl-Hanke-96}) while for $W(u)=\norm{|\nabla u|\,}^2_{\scriptscriptstyle L^2(\Omega)}$ the order-one Tikhonov-Phillips method is obtained. Similarly, the choice of $W(u)=\norm{u}_{\scriptscriptstyle \text{BV}(\Omega)}$ (where $\norm{\cdot}_{\scriptscriptstyle \text{BV}}$ denotes the total variation norm) or $W(u)=\norm{|\nabla u|\,}_{\scriptscriptstyle L^1(\Omega)}$, result in the so called ``bounded variation regularization methods'' (\cite{ref:Acar-Vogel-94}, \cite{ref:Rudin-Osher-Fatemi-1992}). The use of these penalizers is  appropriate when preserving discontinuities or edges is an important matter. The method, however, has as a drawback that it tends to produce piecewise constant approximations and therefore, it will
most likely be inappropriate in regions where the exact solution is smooth (\cite{ref:Chambolle-Lions-97}) producing the so called
``staircasing effect".

In certain types of problems, particularly in those in which it is known that the regularity of the exact solution
is heterogeneous and/or anisotropic, it is reasonable to think that using and spatially adapting two or more penalizers
of different nature could be more convenient. During the last 15 years several regularization methods have been
developed in light of this reasoning. Thus, for instance, in 1997 Blomgren \textit{et al.} (\cite{ref:Blomgren-etal-1997}) proposed
the use of the following penalizer, by using variable $L^p$ spaces:
\begin{equation} \label{eq:penalizante Blomgren}
W(u)=\int_\Omega |\nabla u|^{p(|\nabla u|)} dx,
\end{equation}
where $\underset{u\rightarrow 0^+}{\lim}p(u)=2$, $\underset{u\rightarrow \infty}{\lim}p(u)=1$ and $p$ is a decreasing
function. Thus, in regions where the modulus of the gradient of $u$ is small the penalizer is approximately equal
to $\norm{\abs{\nabla u}}_{L^2(\Omega)}^2$ corresponding to a Tikhonov-Phillips method of order one (appropriate for
restoration in smooth regions). On the other hand, when the modulus of the gradient of $u$ is large, the
penalizer resembles the bounded variation seminorm $\norm{\abs{\nabla u}}_{L^1(\Omega)}$, whose use, as mentioned
earlier, is highly appropriate for border detection purposes. Although this model for $W$ is quite reasonable,
proving basic properties of the corresponding generalized Tikhonov-Phillips functional turns out to be quite
difficult. A different way of combining these two methods was proposed by Chambolle and Lions (\cite{ref:Chambolle-Lions-97}). They
suggested the use of a thresholded penalizer of the form:
\begin{equation*}
    W_\beta(u)=\int_{|\nabla u|\leq \beta}\left|\nabla u\right|^{2} dx +\int_{|\nabla u|> \beta}\left|\nabla u\right|
    dx,
\end{equation*}
where $\beta>0$ is a prescribed threshold parameter. Thus, in regions where borders are more likely to be present ($|\nabla u|>
\beta$), penalization is made with the bounded variation seminorm while a standard order-one Tikhonov-Phillips
method is used otherwise. This model was shown to be successful in restoring images possessing regions with
homogeneous intensity separated by borders. However, in the case of images with non-uniform or highly degraded
intensities, the model is extremely sensitive to the choice of the threshold parameter $\beta$. More recently, penalizers of
the form
\begin{equation}
W(u)=\int_\Omega |\nabla u|^{p(x)} dx,
\end{equation}
for certain functions $p$ with range in $[1,2]$, were studied in \cite{ref:Chen-Levine-Rao-2006} and \cite{ref:Li-Li-Pi-2010}. It is timely to point
out here that all previously mentioned results work only for the case of denoising, i.e. for the case $T=id$.

In this work we propose the use of a model for general restoration problems, which combines, in an appropriate way,
the penalizers corresponding to a zero-order Tikhonov-Phillips method and the bounded variation seminorm. Although
several mathematical issues for this model still remain open, its use in some signal and image restoration problems has already proved
to be very promising. The purpose of this article is to introduce the model, show mathematical results
regarding the existence of the corresponding regularized solutions, and present some results of its application to
signal restoration.

The following Theorem, whose proof can be found in \cite{ref:Acar-Vogel-94} (Theorem 3.1),
guarantees the well-posedness of the unconstrained minimization problem
\begin{equation}\label{eq: funcional AV}
u^\ast\;=\;\underset{u\in L^p(\Omega)}{\rm argmin}\; J(u).
\end{equation}
\begin{thm}\label{teo:existencia AV}
Let $J$ be a  $BV$-coercive functional defined on $L^p(\Omega)$. If $1\leq p<\frac{n}{n-1}$ and $J$ is lower
semicontinuous, then problem (\ref{eq: funcional AV}) has a solution. If $p=\frac{n}{n-1}$, $n\geq 2$ and in
addition  $J$ is weakly lower semicontinuous, then a solutions also exists. In either case, the solution is unique
if $J$ is strictly convex.
\end{thm}
The following theorem, whose proof can also be found in \cite{ref:Acar-Vogel-94} (Theorem 4.1), is very important for the
existence and uniqueness of minimizers of functionals of the form
\begin{equation}\label{eq:funcional-JOTA}
J(u)\;=\;\|Tu-v\|^2+ \alpha J_{\s 0}(u),
\end{equation}
where $\alpha>0$ and $J_{\s 0}(u)$ denotes the bounded variation seminorm given by
\begin{equation}\label{eq:funcional-JOTA-cero}
J_{\s 0}(u)=\underset{\vec\nu\in\mathcal V}{\sup}\int_{\Omega}-u\, \text{div}\vec\nu\, dx,
\end{equation}
with $\mathcal V\doteq\{\vec\nu: \Omega \rightarrow\mathbb R^n \text{ such that } \vec \nu\in C^{\s 1}_{\s
0}(\Omega) \text{ and } |\vec\nu (x)|\leq 1 \,\forall\, x\in\Omega\}.$
\begin{thm}\label{teo:existencia-JOTA} Suppose that $p$ satisfies the restrictions of Theorem \ref{teo:existencia
AV} and $T\chig_\Omega\ne 0$. Then $J(\cdot)$ defined by (\ref{eq:funcional-JOTA}) is $BV$-coercive.
\end{thm}
Note here that (\ref{eq:funcional-JOTA}) is a particular case of (\ref{eq:formulacion W}) with $W(u)=J_{\s 0}(u)$. The
following theorem, whose proof can be found in \cite{ref:Mazzieri-Spies-Temperini-JMAA-2012}, gives conditions
guaranteeing existence and uniqueness of minimizers of (\ref{eq:formulacion W}) for general penalizers $W(u)$. This
theorem will also be very important for our main results in the next section.
\begin{thm}\label{teo:existencia JMAA} Let $\X,\,\Y$ be normed vector spaces,
$T\in\mathcal L(\X,\Y)$, $v\in\Y,$ $\D\subset \X$ a convex set and
$W:\mathcal{D}\longrightarrow \mathbb R$ a functional bounded from
below, $W$-subsequentially weakly lower semicontinuous, and such
that $W$-bounded sets are relatively weakly compact in $\X$. More
precisely, suppose that $W$ satisfies the following hypotheses:
\begin{itemize}
\item \textit{(H1): }$\exists \,\,\gamma\geq 0$ such that $W(u)\geq
-\gamma\hspace{0.3cm} \forall\,u\in\mathcal D$.
\item \textit{(H2):\;}for every $W$-bounded sequence $\{u_n\}\subset\D$  such that
$u_{n}\overset{w}{\rightarrow} u\in\D,$ there exists a subsequence
$\{u_{n_j}\}\subset \{u_n\}$ such that $
W(u)\leq\liminf_{j\rightarrow\infty}W(u_{n_j})$.
\item \textit{(H3):} for every $W$-bounded sequence
$\{u_n\} \subset \mathcal D$ there exist a subsequence
$\{u_{n_j}\}\subset \{u_n\}$ and $u\in \D$ such that
$u_{n_j}\overset{w}{\rightarrow} u$.
\end{itemize}
Then the functional $J_{W,\alpha}(u)\doteq \|Tu-v\|^2+\alpha W(u)$ has a global minimizer on $\mathcal{D}$. If moreover $W$ is
convex and $T$ is injective or if $W$ is strictly convex, then such a minimizer is unique.
\end{thm}
\begin{proof} See  Theorem 2.5 in \cite{ref:Mazzieri-Spies-Temperini-JMAA-2012}.
\end{proof}
\section{Main results}
In this section we will state and prove our main results concerning existence and uniqueness of minimizers of particular
generalized Tikhonov-Phillips functionals with combined spatially-varying $L^2$-BV penalizers. In what follows $\mathcal{M}(\Omega)$ shall denote the set of all real valued measurable functions defined on $\Omega$ and $\widehat{\mathcal{M}}(\Omega)$ the subset of $\mathcal{M}(\Omega)$ formed by those functions with values in $[0,1]$.
\begin{defn} Given $\theta\in \widehat{\mathcal{M}}(\Omega)$ we define the functional $W_{\scriptscriptstyle 0, \theta}(u)$ with values on the extended reals by
\begin{equation}\label{eq:funcional W0theta}
W_{\scriptscriptstyle 0, \theta}(u)\doteq \sup_{\vec\nu\in\mathcal V_{\scriptscriptstyle \theta}}\int_{\Omega}-u\, \text{div}(\theta \vec\nu)\, dx,
\hspace{0.5cm} u \in \mathcal{M}(\Omega)
\end{equation}
where $\mathcal V_{\s \theta}\doteq\{\vec\nu: \Omega\rightarrow\mathbb R^n \text{ such that }
\theta\vec \nu\in C^{\s 1}_{\s 0}(\Omega) \text{ and } |\vec\nu (x)|\leq 1 \,\forall\, x\in\Omega\}.$
\end{defn}
%
\begin{lem}\label{lem: J0 gradiente}
If $u\in C^{\s 1}(\Omega)$ then $W_{\s 0,\theta}(u)=\|\theta \abs{\nabla u}\|_{L^{\s
1}(\Omega)}$.
\end{lem}
\begin{proof}

Let $u\in C^1(\Omega).$ For all $\vec\nu \in \mathcal V_{\s \theta}$ it follows easily that
\begin{align}\label{eq: desig-thetagrad}
\nonumber \ds\int_{\Omega}-u\, \text{div}(\theta \vec\nu)\, dx&=\ds\int_{\Omega}\nabla u \cdot \theta\vec \nu \,
dx - \ds\int_{\delta\Omega}(u\theta\vec\nu \cdot \vec n )\, dS  \\ \nonumber
&=\int_{\Omega}\nabla u\cdot\theta\vec\nu\;dx \qquad (\text{since } \theta \vec\nu|_ {\delta\Omega}=0)   \\ \nonumber
&\leq \ds\int_{\Omega}\abs{\theta\nabla u} \abs{\vec \nu}\, dx\\
&\leq \ds\int_{\Omega}\abs{\theta\nabla u} \, dx \quad (\text{since} \abs{\vec\nu(x)}\leq 1),
\end{align}
where $\vec n$ denotes the outward unit normal to $\delta\Omega$. Taking supremum over $\vec \nu \in \mathcal
V_{\s\theta}$ it follows that
$$W_{\scriptscriptstyle 0, \theta}(u) \leq \|\theta \abs{\nabla u} \|_{ L^1(\Omega)}.$$
For the opposite inequality, define $\vec\nu_\ast(x)\doteq \begin{cases}
 \frac{\nabla u(x)}{\abs{\nabla u(x)}}, &\text{if }\abs{\nabla u(x)}\neq 0,\\
0, &\text{if } \abs{\nabla u(x)}= 0,
\end{cases}$. Then one has that $\abs{\vec\nu_\ast(x)}\leq 1\; \forall \, x \in \Omega$ and $\vec\nu_\ast\in C(\Omega, \mathbb{R}^n)$ since $u\in C^{\s 1}(\Omega).$ Also,
$$\ds\int_{\Omega}(\nabla u \cdot \theta\vec\nu_\ast)\, dx = \ds\int_{\Omega}\abs{\theta\nabla u} \, dx.$$ By convolving $\vec\nu_\ast$ with an appropriately chosen function $\varphi \in C^{\s \infty}_{\s 0}(\Omega, \mathbb{R}^n)$, we can obtain a function $\vec\nu \in \mathcal V_{\s \theta} \cap C^{\s \infty}_{\s 0}(\Omega, \mathbb{R}^n)$ for which the left hand side of (\ref{eq: desig-thetagrad}) is arbitrarily close to $\ds\int_{\Omega}\abs{\theta\nabla u} \, dx.$ Then taking supremum over $\vec\nu \in \mathcal V_\theta$ we have that $$W_{\scriptscriptstyle 0, \theta}(u) \geq \|\theta \abs{\nabla u} \|_{ L^1(\Omega)}.$$ Hence $W_{\scriptscriptstyle 0, \theta}(u) = \|\theta \abs{\nabla u} \|_{ L^1(\Omega)},$ as we wanted to prove.
\end{proof}
\noindent Observation:  From de density of $C^{\s 1}(\Omega)$ in $ W^{\s 1,1}(\Omega)$ it follows that Lemma
\ref{lem: J0 gradiente} holds for every $ u\, \in W^{\s 1,1}(\Omega)$.
%


\begin{rem}
For any $\theta \in \widehat{\mathcal{M}}(\Omega)$, it follows easily that
\begin{equation}\label{eq:desigW-J}
W_{\s 0,\theta}(u)\leq J_{\s 0}(u), \quad \forall \,u \in \mathcal{M}(\Omega).
\end{equation}
In fact, for any $\vec\nu\in \mathcal V_{\s \theta}$ and for any $u \in \mathcal{M}(\Omega)$ we have that
\begin{align}\label{eq:desig-remark}
\int_{\Omega}-u\, \text{div}(\theta \vec\nu)\, dx&\leq \underset {\vec\nu \in \mathcal{V}} {\sup}\int_{\Omega}-u\, \text{div}\,\vec\nu\, dx \\ \nonumber
&=J_{\s 0}(u),
\end{align}
where inequality (\ref{eq:desig-remark}) follows from the fact that $\theta\vec\nu \in \mathcal V$ (since
$|\theta(x)|\leq 1 \;\forall x\in\Omega$).  By taking supremum for $\vec\nu \in \mathcal{V}_\theta$ inequality
(\ref{eq:desigW-J}) follows.
\end{rem}

Although inequality (\ref{eq:desigW-J}) is important by itself since it relates the functionals $W_{\s 0,\theta}$ and $J_{\s 0}$,
in order to be able to use the known coercitivity properties of $J_{\s 0}$ (see \cite{ref:Acar-Vogel-94}), an inequality of the opposite type is highly desired. That is, we would like to show that, under certain conditions on $\theta(\cdot)$, there exists a constant $C=C(\theta)$ such that $W_{\s 0,\theta}(u)\geq CJ_{\s 0}(u)$ for all $u \in \mathcal{M}(\Omega)$. The following theorem provides sufficient conditions on $\theta$ assuring such an inequality.
\begin{thm}\label{teo:desig-WgJ}
Let $\theta \in \widehat{\mathcal{M}}(\Omega)$ be such that $\frac1\theta\in L^\infty(\Omega)$ and
let $J_{\s 0}$, $W_{\s 0,\theta}$ be the functionals defined in (\ref{eq:funcional-JOTA-cero}) and
(\ref{eq:funcional W0theta}), respectively. Then $J_{\s 0}(u) \le \|\frac1\theta\|_{\s L^\infty(\Omega)}\, W_{\s 0,\theta}(u)$
for all $u \in \mathcal{M}(\Omega)$.
\end{thm}
\begin{proof}
Let $u \in \mathcal{M}(\Omega)$ and $K_{\s\theta}\doteq\|\frac1\theta\|_{\s L^\infty(\Omega)}$. Then for all $\vec\nu \in \mathcal{V}$
\begin{align*}
\int_{\Omega}-u\, \text{div}\,\vec\nu\, dx&= K_{\s\theta} \int_{\Omega}-u\, \text{div}\left(
\frac{\theta\vec\nu}{K_{\s\theta}\theta}\right)\, dx \\
&\leq K_{\s\theta}\;\underset {\vec\omega \in \mathcal{V}_{\s \theta}} {\sup}\int_{\Omega}-u\,
\text{div}\,(\theta\,\vec\omega)\,
dx \\
&=K_{\s\theta}\;W{\s 0,\theta}(u),
\end{align*}
where the last inequality follows from the fact that $\frac{\vec\nu}{K_{\s\theta}\,\theta}\in \mathcal{V}_{\s \theta}$
since $K_{\s \theta}\geq 1, |K_{\s \theta} \theta(x)|\geq 1 \, \forall \, x \in\Omega$ and $\vec \nu \in \mathcal V$. Then, taking supremum for $\vec\nu\in\mathcal{V}$ we conclude that  $J_{\s 0}(u)
\le K_{\s\theta}\, W_{\s 0,\theta}(u)$.
\end{proof}
The following lemma will be of fundamental importance for proving several of the upcoming results.
\begin{lem}\label{lem:wls-W}
The functional $W_{\s 0,\theta}$ defined by (\ref{eq:funcional W0theta}) is weakly lower semicontinuous
with respect to the $L^p$ topology, $\forall\, p\in[1,\infty)$.
\end{lem}
\begin{proof}
Let $p\in [1,\infty)$, $\{u_n\}\subset L^p(\Omega)$ and $u\in L^p(\Omega)$ be such that $u_n\overset{w}{\rightarrow} u$. Let $\vec{\nu}_\ast\in \mathcal V_{\s \theta}$ and $q$ the conjugate dual of $p$. Since $\theta\vec \nu_\ast\in C^{\s 1}_{\s 0}(\Omega)$, it follows that $\text{div}(\theta\vec\nu_\ast)$
is uniformly bounded on $\Omega$ and therefore, $\text{div}(\theta\vec\nu_\ast)\in L^{\infty}(\Omega)\subset L^q(\Omega).$ Then, from the weak convergence of $u_n$ it follows that $\underset{n\to \infty}{\lim} \int_{\Omega}-u_n\, \text{div}(\theta \vec\nu_\ast)\, dx = \ds \int_{\Omega}-u\, \text{div}(\theta \vec\nu_\ast)\, dx$.

Hence $\ds\int_{\Omega}-u\, \text{div}(\theta \vec\nu_\ast)\, dx=\underset{n\to \infty}{\lim}  \int_{\Omega}-u_n\, \text{div}(\theta \vec\nu_\ast)\, dx\leq \liminf_{n\rightarrow\infty} \sup_{\vec \nu \in\mathcal V_{\s \theta}}\int_{\Omega}-u_n\, \text{div}(\theta \vec\nu)\, dx =\break \liminf_{n\rightarrow\infty}W_{\s 0,\theta}(u_n)$. Thus $\forall \, \vec\nu_\ast \in \mathcal V_{\s \theta}$
$$\ds\int_{\Omega}-u\, \text{div}(\theta \vec\nu_\ast)\, dx\leq \liminf_{n\rightarrow\infty}W_{\s 0,\theta}(u_n).$$
Taking supremum over all $\vec\nu_\ast \in \mathcal V_{\s \theta}$ it follows that $W_{\s 0,\theta}(u)\leq
\underset{n\rightarrow\infty}{\liminf} \,W_{\s 0,\theta}(u_n).$
\end{proof}
We are now ready to present several results on existence and uniqueness of minimizers of generalized
Tikhonov-Phillips functionals with penalizers involving spatially varying combinations of the $L^2$-norm and of the
functional $W_{\s 0,\theta}$, under different hypotheses on the function $\theta$.
\begin{thm}\label{teo:existencia_unicidad ep1-ep2}
Let $\Omega\subset\mathbb{R}^n$ be a bounded open convex set with Lipschitz boundary, $\X=L^2(\Omega)$, $\Y$ a normed
vector space, $T\in \mathcal L(\X, \Y)$, $v\in\Y$, $\alpha_{\scriptscriptstyle 1},\, \alpha_{\scriptscriptstyle 2}$
positive constants, $\theta \in \widehat{\mathcal{M}}(\Omega)$ and $F_{\s\theta}$ the functional defined by
\begin{equation}\label{eq:funcional J}
F_{\s\theta}(u)\doteq \norm{Tu-v}^2_{\Y}+\alpha_{\scriptscriptstyle 1}\|\sqrt{1-\theta}\,u\|^2_{\scriptscriptstyle
L^2(\Omega)}+\alpha_{\scriptscriptstyle 2}\,W_{\s 0, \theta}(u),\hspace{0.5cm} u\in \mathcal{D}\doteq L^2(\Omega).
\end{equation}
If there exists $\varepsilon_{\s 2}\in \mathbb{R}$, such that
$\theta(x)\leq\varepsilon_{2}<1$ for a.e. $x \in \Omega$, then the functional (\ref{eq:funcional J})
has a unique global minimizer $u^*\in L^2(\Omega)$. If moreover there exists $\varepsilon_{\s 1}\in \mathbb{R}$ such that
$0<\varepsilon_{\s 1}\leq \theta(x)$ for a.e. $x \in \Omega$, then $u^*\in BV(\Omega)$.

\end{thm}
\begin{proof}
By virtue of Theorem \ref{teo:existencia JMAA} it is sufficient to show that the functional $$W(u)\doteq \alpha_{\scriptscriptstyle 1}\|\sqrt{1-\theta}\,u\|^2_{\scriptscriptstyle L^2(\Omega)}+\alpha_{\scriptscriptstyle 2}\,W_{\s 0, \theta}(u),\, u\in L^2(\Omega)$$
satisfies hypotheses \textit{(H1)}, \textit{(H2)} and \textit{(H3)}. Clearly \textit{(H1)} holds with $\gamma=0$.

To prove (\textit{H2}) let $\{u_n\}\subset L^2(\Omega)$ such that $u_n \overset{w}{\longrightarrow} u\in L^2(\Omega)$ and $W(u_n)\leq c_1<\infty$. We want to show that $W(u)\leq \underset{n\rightarrow\infty}{\liminf} \,W(u_n).$ Since
$\sqrt{1-\theta}\in L^\infty(\Omega)$ one has $\sqrt{1-\theta}\,u_n\overset{w}{\longrightarrow} \sqrt{1-\theta}\,u$.

The condition $\theta(x)\leq \varepsilon_2<1$ for a.e. $x \in \Omega$, clearly implies that $\|\sqrt{1-\theta}\cdot\|_{\s L^2(\Omega)}$ is a norm.
Then, from the weak lower semicontinuity of  $\|\sqrt{1-\theta}\cdot\|^2_{\s L^2(\Omega)},$ it follows that
\begin{equation}\label{eq:dsi norma 2}
\|\sqrt{1-\theta}\,u\|^2_{L^2(\Omega)}\leq \liminf_{n\rightarrow\infty}\|\sqrt{1-\theta}\,u_n\|^2_{L^2(\Omega)}.
\end{equation}
On the other hand, from the weak lower semicontinuity of $\W$ in $L^2(\Omega)$ (see Lemma \ref{lem:wls-W}) it follows that
\begin{equation}\label{eq:dsi J0}
 \W(u)\leq\liminf_{n\rightarrow\infty}\W(u_n).
\end{equation}
From (\ref{eq:dsi norma 2}) and (\ref{eq:dsi J0}) we then conclude that
\begin{align*}
W(u)&=\alpha_{\scriptscriptstyle 1}\|\sqrt{1-\theta}\,u\|^2_{\scriptscriptstyle L^2(\Omega)}+\alpha_{\scriptscriptstyle 2}\,W_{\s 0, \theta}(u)\\
&\leq \alpha_{\scriptscriptstyle 1}\liminf_{n\rightarrow\infty}\|\sqrt{1-\theta}\,u_n\|^2_{L^2(\Omega)}+\alpha_{\scriptscriptstyle 2} \liminf_{n\rightarrow\infty}\W(u_n)\\
&\leq \liminf_{n\rightarrow\infty}\left(\alpha_{\scriptscriptstyle 1}\|\sqrt{1-\theta}\,u_n\|^2_{L^2(\Omega)}+\alpha_{\scriptscriptstyle 2} \W(u_n)\right)\\
&= \liminf_{n\rightarrow\infty}W(u_n),
\end{align*}
what proves (\textit{H2}).

To prove (\textit{H3}) let $\{u_n\}\subset L^2(\Omega)$ be such that $W(u_n)\leq c_1 <\infty,\,\,\forall \,n$. We want to show that there exist $\{u_{n_j}\}\subset\{u_n\}$ and $u\in L^{2}(\Omega)$ such that $u_{n_j}\overset{w}{\longrightarrow} u$. For this note that
\begin{align}\label{eq:desig-H3}
(1-\varepsilon_{\s 2})\|u_n\|^2_{L^2(\Omega)}\leq \|\sqrt{1-\theta}\,u_n\|^2_{L^2(\Omega)}\leq W(u_n)\leq c_1.
\end{align}
Thus $\|u_n\|_{\s L^2(\Omega)}$ is uniformly bounded and therefore there exist $\{u_{n_j}\}\subset\{u_n\}$ and $u^*\in L^2(\Omega)$ such that $u_{n_j}\overset{w}{\longrightarrow} u^*$. Hence, by Theorem \ref{teo:existencia JMAA}, the functional $F_{\s\theta}(u)$ given by (\ref{eq:funcional J}) has a global minimizer $u^* \in L^2(\Omega).$ The condition $\theta(x)\leq \varepsilon_2<1$ for a.e. $x \in \Omega$ clearly implies the strict convexity of $F_{\s \theta}$ and therefore the uniqueness of such a global minimizer.

To prove the second part of the theorem, assume further that there exists $\varepsilon_1>0$ such that $\theta(x)\geq \varepsilon_1$ for a.e. $x \in \Omega.$ Following the proof of Theorem 5.1 in \cite{ref:Mazzieri-Spies-Temperini-JMAA-2012}, it suffices to show that under this additional hypothesis the weak limit $u$ in \textit{(H3)} above belongs to $BV(\Omega).$ For this note that from  (\ref{eq:desig-H3}) it follows that there exist $c_2<\infty$ such that
\begin{equation}\label{eq:desig2-H3}
\|u_n\|_{\s L^1(\Omega)}\leq c_2 \quad \forall \;n.
\end{equation}
Also, by Theorem \ref{teo:desig-WgJ} $W_{\scriptscriptstyle 0, \theta}(u)\geq \varepsilon_{\s 1}J_{\s 0}(u)\,\,\forall\, u\in \mathcal M(\Omega).$ This, together with    (\ref{eq:desig2-H3}) implies that
\begin{align*}
\|u_n\|_{BV(\Omega)}= \|u_n\|_{L^1(\Omega)}+ J_{\s 0}(u_n)\leq c_2+\frac{W_{\s 0,\theta}(u_n)}{\varepsilon_1}\leq c_3<\infty\,\,\forall \,n,
\end{align*}
where the previous to last inequality follows from the uniform boundedness of $W_{\s 0,\theta}(u_n)$, which, in turn, follows from the uniform boundedness of $W(u_n)$.  Hence the fact that the weak limit in (\textit{H3}) is in $BV(\Omega$) follows from the compact embedding of $BV(\Omega$) in $L^2(\Omega)$. This result is an extension of the Rellich-Kondrachov Theorem which can be found, for instance, in \cite{refb:Adams-1975} and \cite{refb:Attouch-Buttazzo-Michaille-2006}.
\end{proof}

\begin{rem}
Note that if $\theta(x)=0\; \forall\, x\in\Omega$, then $W(u)=\|u\|_{L^2(\Omega)}^2$ and  $F_{\s \theta}$  as defined in (\ref{eq:funcional J}) is the
classical Tikhonov-Phillips functional of order zero. On the other hand, if $\theta(x)=1\; \forall\, x\in\Omega$
then $W(u)=J_{\s 0}(u)$ and $F_{\s \theta}$ has a global minimizer provided that $T \chig_\Omega\neq 0$. If moreover $T$ is injective then
such a global minimizer is unique. All these facts follow immediately from Theorems 3.1 and 4.1 in
\cite{ref:Acar-Vogel-94}.
\end{rem}

\begin{thm}\label{teo:L1-Linf}
Let $\Omega\subset\mathbb{R}^n$ be a bounded open convex set with Lipschitz boundary, $\X=L^2(\Omega)$, $\Y$ a normed
vector space, $T\in \mathcal L(\X, \Y)$, $v\in\Y$, $\alpha_{\scriptscriptstyle 1},\, \alpha_{\scriptscriptstyle 2}$
positive constants and $\theta \in \widehat{\mathcal{M}}(\Omega)$ such that $\frac{1}{1-\theta}\in L^{\s 1}(\Omega)$ and $\frac{1}{\theta}\in L^{\s \infty}(\Omega)$. Then the functional (\ref{eq:funcional J}) has a unique global minimizer $u^*\in BV(\Omega)$.
\end{thm}
\begin{proof}
Let us consider the functional $$W(u)\doteq \alpha_{\scriptscriptstyle
1}\|\sqrt{1-\theta}\,u\|^2_{\scriptscriptstyle L^2(\Omega)}+\alpha_{\scriptscriptstyle 2}\,W_{\s 0, \theta}(u),\,
u\in L^2(\Omega).$$
By virtue of Theorems \ref{teo:existencia JMAA} and \ref{teo:existencia_unicidad ep1-ep2} and the compact embedding of $BV(\Omega)$ in $L^2(\Omega)$, it suffices to show that $W(\cdot)$ satisfies \textit{(H1)} and \textit{(H2)} and that every $W$-bounded sequence is also $BV$-bounded. Clearly $W(\cdot)$ satisfies \textit{(H1)} with $\gamma=0$. That it satisfies \textit{(H2)} follows immediately from the fact that the condition $\frac{1}{1-\theta}\in L^{\s 1}(\Omega)$ implies that
$\|\sqrt{1-\theta}\cdot\|_{\s L^2(\Omega)}$ is a norm.

Now, let $\{u_n\}\subset L^2(\Omega)$ be a $W$-bounded sequence, i.e. such that $W(u_n)\leq c <\infty,\,\,\forall
\,n$. We will show that $\{u_n\}$ is $BV$-bounded. Since $W(u_n)$ is uniformly bounded, there exist $K<\infty$ such
that $\|\sqrt{1-\theta} \,u_n\|_{\s L^2(\Omega)}\leq K \; \forall \,n.$ From this and the fact that
$\frac{1}{1-\theta}\in L^{\s 1} (\Omega)$ it follows that
\begin{align}\label{eq:desig3-H3}\nonumber
\|u_n\|_{\s L^1(\Omega)}&=\int_\Omega \frac{1}{\sqrt{1-\theta}}\;\sqrt{1-\theta}\,|u_n|\, dx\\ \nonumber
&\leq \left(\int_\Omega \frac{1}{1-\theta}\,dx\right)^{\frac 1 2}\,\left(\int_\Omega
(1-\theta)\,u_n^2\,dx\right)^{\frac 1 2}\\ \nonumber
&=\left\|\frac{1}{1-\theta}\right \|_{\s L^1(\Omega)}^{\frac 1 2}\, \|\sqrt{1-\theta}\,u_n\|_{\s L^2(\Omega)}\\
&\leq K\,\left\|\frac{1}{1-\theta}\right \|_{\s L^1(\Omega)}^{\frac 1 2}<\infty\quad \forall\, n.
\end{align}

On the other hand from Theorem \ref{teo:desig-WgJ} $J_{\s 0}(u)\leq W_{\scriptscriptstyle 0, \theta}(u)\, \left\|\frac 1 {\theta}\right\|_{\s L^\infty(\Omega)}
\, \forall\, u\in L^{\s 2}(\Omega)$. Since $\frac 1 \theta \in L^{\s \infty}(\Omega)$ and $W_{\scriptscriptstyle 0,
\theta}(u_n)$ is uniformly bounded, it then follows that there exists $C<\infty$ such that
\begin{equation}\label{eq:desig4-H3}
J_{\s 0}(u_n)\leq C \; \forall \,n.
\end{equation}
From (\ref{eq:desig3-H3}) and (\ref{eq:desig4-H3}) it follows that
\begin{align*}
\|u_n\|_{BV(\Omega)}= \|u_n\|_{L^1(\Omega)}+ J_{\s 0}(u_n)\leq K\,\left\|\frac{1}{1-\theta}\right
\|_{\s L^1(\Omega)}^{\frac 1 2}+C<\infty\,\,\forall \, n.
\end{align*}
Hence $\{u_n\}$ is $BV$-bounded. The existence of a global minimizer of functional (\ref{eq:funcional J}) belonging to $BV(\Omega)$ then follows.
Finally note that the condition $\frac{1}{1-\theta}\in L^{\s 1}(\Omega)$ implies the strict convexity of
 $F_{\s \theta}$ and therefore the uniqueness of the global minimizer.
\end{proof}
\begin{rem}
Note that the condition $\frac{1}{1-\theta}\in L^{\s 1}(\Omega)$ in Theorem \ref{teo:L1-Linf} is weaker than the condition $\theta(x)\leq\varepsilon_{2}<1$ for a.e. $x \in \Omega$ of Theorem \ref{teo:existencia_unicidad ep1-ep2}. While the latter suffices to guarantee the existence of a global minimizer in $L^{\s 2}(\Omega)$, the former does not. However this weaker condition $\frac{1}{1-\theta}\in L^{\s 1}(\Omega)$ together with the condition $\frac1\theta \in  L^\infty(\Omega)$ are enough for guaranteing not only the existence of a unique global minimizer, but also the fact that such a minimizer belongs to $BV(\Omega)$.
\end{rem}
It is timely to note that in both Theorems \ref{teo:existencia_unicidad ep1-ep2} and \ref{teo:L1-Linf}, the function $\theta$ cannot assume the extreme values 0 and 1 on a set of positive measure. In some cases a pure $BV$ regularization in some regions and a pure $L^2$ regularization in others may be desired, and therefore that restraint on the function $\theta$ will turn out to be inappropriate. In the next three theorems we introduce different conditions which allow the function $\theta$ to take the extreme values on sets of positive measure.
\begin{thm}
Let $\Omega\subset\mathbb{R}^n$ be a bounded open convex set with Lipschitz boundary, $\X=L^2(\Omega)$, $\Y$ a normed
vector space, $T\in \mathcal L(\X, \Y)$, $v\in\Y$, $\alpha_{\scriptscriptstyle 1},\, \alpha_{\scriptscriptstyle 2}$
positive constants,  $\theta \in \widehat{\mathcal{M}}(\Omega)$ and $\Omega_{\s 0}\doteq \{x\in\Omega \text{ such that }\theta(x)=0\}$. If
$\frac{1}{\theta}\in L^{\s \infty}(\Omega_{\s \,0}^{\s \,c})$ and $\frac{1}{1-\theta}\in L^{\s 1}(\Omega_{\s \,0}^{\s \,c})$
then functional (\ref{eq:funcional J}) has a unique global minimizer $u^*\in L^2(\Omega)\cap BV(\Omega_{\s \,0}^{\s \,c}).$
\end{thm}
\begin{proof}
Under the hypotheses of the theorem the functional $W(u)$ can be written as
\begin{equation}\label{eq:W}
W(u)=\alpha_{\s 1}\|u\|^2_{\scriptscriptstyle L^2(\Omega_{\s \,0})}+\alpha_{\s 1}\|\sqrt{1-\theta}\,u\|^2_{\s L^2(\Omega_{\s \,0}^{\s \,c})}+\alpha_{\s 2} \sup_{\vec\nu\in\mathcal V_{\s \theta}}\int_{\Omega_{\s \,0}^{\s \,c}}-u|_{\Omega_{\,0}^{\,c}}\, \text{div}(\theta \vec\nu)\, dx.
\end{equation}

Just like in Theorem \ref{teo:L1-Linf} it follows easily that $W(\cdot)$ satisfies \textit{(H1)} and \textit{(H2)}.

Let now $\{u_n\}\subset L^2(\Omega)$ be a $W$-bounded sequence. From (\ref{eq:W}) we conclude that there exist $u^*_1\in L^2(\Omega_{\s 0})$ and a subsequence $\{u_{n_j}\}\subset\{u_n\}$ such that $u_{n_j}|_{\s {\Omega_{\,0}}} \overset{w-L^2(\Omega_0)}{\longrightarrow} u^*_1$. On the other hand from the uniform boundedness of $\sup_{\vec\nu\in\mathcal V_{\s \theta}}\int_{\Omega_{\s \,0}^{\s \,c}}-u_{n_j}|_{\Omega_{\,0}^{\,c}}\, \text{div}(\theta \vec\nu)\, dx$, by using Theorem \ref{teo:desig-WgJ} with $\Omega$ replaced by $\Omega_{\s \,0}^{\s \,c}$, it follows that there exists a constant $C\leq \infty$ such that $J_{\s 0}(u_{n_j}|_{\Omega_{\,0}^{\,c}})\leq C$ for all $n_j$. Also, from (\ref{eq:W}) and the hypothesis that $\frac{1}{1-\theta}\in L^{\s 1}(\Omega_{\s \,0}^{\s \,c})$, it can be easily proved that the sequence $\{u_n\}$ is uniformly bounded in $ L^{\s 1}(\Omega_{\s \,0}^{\s \,c}).$ Hence $\left\{u_{n_j}|_{\Omega_{\,0}^{\,c}}\right\}$ is uniformly $BV$-bounded. By using the compact embedding of $BV(\Omega_{\s \,0}^{\s \,c})$ in $L^2(\Omega_{\s \,0}^{\s \,c})$ it follows that there exist a subsequence $\{u_{n_{j_k}}\}$ of $\{u_{n_j}\}$ and $u^*_{\s 2}\in BV(\Omega_{\s \,0}^{\s \,c})$ such that $u_{n_{j_k}}\overset{w-L^2(\Omega_{\s \,0}^{\s \,c})}{\longrightarrow}u^*_{\s 2}.$

Let us define now
\begin{equation*}
\hat{u}_{\s 1}(x)\doteq\left\{
  \begin{array}{ll}
   u^*_{\s 1}(x) , & \hbox{if $x\in \Omega_{\s \,0}$,} \\
    0, & \hbox{if $x\in \Omega_{\s \,0}^{\s \,c}$,}
  \end{array}
\right.
\end{equation*}
\begin{equation*}
\hat{u}_{\s 2}(x)\doteq\left\{
  \begin{array}{ll}
   u^*_{\s 2}(x) , & \hbox{if $x\in \Omega_{\s \,0}^{\s \,c}$,} \\
    0, & \hbox{if $x\in \Omega_{\s \,0}$,}
  \end{array}
\right.
\end{equation*}
and $u^*\doteq \hat{u}_{\s 1}+ \hat{u}_{\s 2}.$ Then one has that $u^*\in L^2(\Omega)$, $u^*|_{\s {\Omega_{\,0}^{\,c}}}=u_2^*\in BV(\Omega_{\s \,0}^{\s \,c})$ and $u_{n_{j_k}}\overset{w-L^2(\Omega)}{\longrightarrow}u^*.$

The existence of a global minimizer of functional (\ref{eq:funcional J}) then follows immediately from Theorem $\ref{teo:existencia JMAA}$. Uniqueness is a consequence of the fact that the hypothesis $\frac{1}{1-\theta}\in L^{\s 1}(\Omega_{\s \,0}^{\s \,c})$ implies that $\|\sqrt{1-\theta}\cdot\|_{\s L^2(\Omega_{\s \,0}^{\s \,c})}$ is a norm.
\end{proof}
\begin{thm}\label{teo omega1}
Let $n\leq 2$, $\Omega\subset\mathbb{R}^n$ be a bounded open convex set with Lipschitz boundary, $\X=L^2(\Omega)$, $\Y$ a
normed vector space, $T\in \mathcal L(\X, \Y)$, $v\in\Y$, $\alpha_{\scriptscriptstyle 1},\,
\alpha_{\scriptscriptstyle 2}$ positive constants. Let  $\theta \in \widehat{\mathcal{M}}(\Omega)$ and $\Omega_{\s 1}\doteq
\{x\in\Omega \text{ such that } \theta(x)=1\}$. If $\frac{1}{\theta}\in L^{\s \infty}(\Omega_{\, \s 1}^{\,\s c})$,
$\frac{1}{1-\theta}\in L^{\s 1}(\Omega_{\,\s 1}^{\,\s c})$ and $T\chig_{\s \Omega}\neq 0$, then the
functional (\ref{eq:funcional J}) has a global minimizer $u^*\in L^2(\Omega) \cap BV(\Omega_{\s 1}^{\s c})$.  If moreover $\mathcal{N}(T)$ does not contain functions vanishing on $\Omega_1$, i.e. if $Tu=0$ implies $u|_{\s \Omega_1}\ne 0$, then such a global minimizer is unique.
\end{thm}

\begin{proof}
We will prove that under the hypotheses of the theorem, the functional $F_{\s\theta}(\cdot)$ defined by (\ref{eq:funcional J}) is weakly lower semicontinuous with respect to the $L^2(\Omega)$ topology and $BV$-coercive.

First note that under the hypotheses of the theorem we can write
\begin{equation}\label{eq:F-theta-omega1}
F_{\s \theta}(u)= \|Tu-v\|^2_\Y+\alpha_{\scriptscriptstyle 1}\|\sqrt{1-\theta}\,u\|^2_{\scriptscriptstyle L^2(\Omega_{\s \, 1}^{\,c})}+\alpha_{\scriptscriptstyle 2}\,W_{\s 0, \theta}(u).
\end{equation}
Since $\frac{1}{1-\theta}\in L^{\s 1}(\Omega_{\,\s 1}^{\,\s c})$, it follows that $\|\sqrt{1-\theta}\,\cdot\|_{L^2(\Omega_1^c)}$ is a norm in $L^2(\Omega_1^c)$ and therefore it is weakly lower semicontinuous.  The weak lower semicontinuity of $F_{\s\theta}(\cdot)$ then follows immediately from this fact, from Lemma \ref{lem:wls-W} and from the weak lower semicontinuity of the norm in $\Y$.

For the $BV$-coercitivity, note that
\begin{align}\label{eq:lalala}\nonumber
\|Tu-v\|^2+\alpha_2 J_{\s 0}(u) &\le \|Tu-v\|^2 +\alpha_2\left\|\frac1\theta\right\|_{\s L^\infty(\Omega_1^c)} W_{\s 0,\theta}(u) \quad(\text{from Theorem (\ref{teo:desig-WgJ})}) \\ \nonumber
&\le \|Tu-v\|^2 +\alpha_2\left\|\frac1\theta\right\|_{\s L^\infty(\Omega_1^c)} W_{\s 0,\theta}(u) +\alpha_1\|\sqrt{1-\theta}\,u\|^2_{\s L^2(\Omega_1^c)}\\
&\le \left\|\frac1\theta\right\|_{\s L^\infty(\Omega_1^c)} F_{\s \theta}(u)\qquad\qquad(\text{since } \left\|\theta^{-1}\right\|_{\s L^\infty(\Omega_1^c)}\ge 1).
\end{align}
Now, since  $T\chig_{\s \Omega}\neq 0$, by Theorem \ref{teo:existencia-JOTA} the functional $J(u)\doteq\|Tu-v\|^2+\alpha_2J_{\s 0}(u)$ is $BV$-coercive on $L^2(\Omega)$. From this and inequality (\ref{eq:lalala}) it follows that $F_{\s \theta}(\cdot)$ is also $BV$-coercive.
The existence of a global minimizer $u^* \in L^{\s 2}(\Omega)$ then follows from Theorem \ref{teo:existencia AV}. Since $F_{\s \theta}(u^*)<\infty$ it follows that both $\|u^*\|_{\s L^1(\Omega_1^c)}$ and $W_{\s 0,\theta}(u^*)$ are finite. The fact that  $u^*$ is of bounded variation on $\Omega_1^c$ then follows from Theorem \ref{teo:desig-WgJ}.
Finally, if $\mathcal{N}(T)$ does not contain functions vanishing on $\Omega_1$ then it follows easily that $F_{\s \theta}(u)$ is strictly convex and therefore such a global minimizer is unique.
\end{proof}

\begin{thm}
Let $\Omega\subset\mathbb{R}^n$, $n\leq 2$ be a bounded open convex set with Lipschitz boundary, $\X=L^{\s 2}(\Omega)$, $\Y$ a
normed vector space, $T\in \mathcal L(\X, \Y)$, $v\in\Y$, $\alpha_{\scriptscriptstyle 1},\,
\alpha_{\scriptscriptstyle 2}$ positive constants. Let  $\theta \in \widehat{\mathcal{M}}(\Omega)$, $\Omega_{\s 0}\doteq \{x\in\Omega \text{ such that } \theta(x)=0\}$ and $\Omega_{\s 1}\doteq
\{x\in\Omega \text{ such that } \theta(x)=1\}$. If $\frac{1}{\theta}\in L^{\s \infty}(\Omega_{\, \s 0}^{\,\s c})$,
$\frac{1}{1-\theta}\in L^{\s \infty}(\Omega_{\,\s 1}^{\,\s c})$ and $\mathcal{N}(T)$ does not contain functions vanishing on $\Omega_1$, i.e. if $Tu=0$ implies $u|_{\s \Omega_1}\ne 0$, then functional (\ref{eq:funcional J}) has a unique global minimizer $u^*\in L^2(\Omega) \cap BV(\Omega_{\s 1}^{\s c}\cap \Omega_{\s 0}^{\s c})$.
\end{thm}
\begin{proof}
For the existence of a global minimizer
it is sufficient to prove that the functional $F_{\theta}$ defined by  (\ref{eq:F-theta-omega1}) is weakly lower semicontinuos and $L^{\s 2}(\Omega)$-coercive. For this, note that

\begin{equation}\label{eq:F-theta-omega01}
F_{\s \theta}(u)= \|Tu-v\|^2_\Y+\alpha_{\scriptscriptstyle 1}\|\sqrt{1-\theta}\,u\|^2_{\scriptscriptstyle L^2(\Omega_{\s \, 1}^{\,c})}+\alpha_{\scriptscriptstyle 2}\,\sup_{\vec\nu\in\mathcal V_{\s \theta}}\int_{\Omega_{\s \,0}^{\s \,c}}-u\, \text{div}(\theta \vec\nu)\, dx.
\end{equation}

The weak lower semicontinuity of $F_{\s\theta}(\cdot)$ follows from Lemma \ref{lem:wls-W} and the weak lower semicontinuity of the norms in $\Y$ and $\|\sqrt{1-\theta}\,\cdot\|_{L^2(\Omega_1^c)}$.

We shall now prove that $F_{\s\theta}(\cdot)$ is $L^{\s 2}(\Omega)$-coercive. For that, assume $\{u_n\}$ is a sequence in $L^{\s 2}(\Omega)$ such that
$\|u_n\|_{\s L^{\s 2}(\Omega)}\rightarrow \infty$. Then either $\|u_n\|_{\s L^{\s 2}(\Omega_1^c)}\rightarrow \infty$ or $\|u_n\|_{\s L^{\s 2}(\Omega_1)}\rightarrow \infty$. If $\|u_n\|_{\s L^{\s 2}(\Omega_1^c)}\rightarrow \infty$, then the hypothesis $\frac{1}{1-\theta}\in L^{\s \infty}(\Omega_{\,\s 1}^{\,\s c})$ implies that $\|\sqrt{1-\theta}\,u\|^2_{\scriptscriptstyle L^2(\Omega_{\s \, 1}^{\,c})}\,\rightarrow \infty$ and therefore $F_{\s \theta}(u_n)\rightarrow \infty$. Suppose now that $\|u_n\|_{\s L^{\s 2}(\Omega_1)}\rightarrow \infty$ and without loss of generality assume that $\|u_n\|_{\s L^{\s 2}(\Omega_{\s 1}^{\s c})} \le C<\infty$. Then due to the compact embedding $BV(\Omega_{\s 1})\hookrightarrow L^{\s 2}(\Omega_{\s 1})$
it follows that $\|u_n\|_{\s BV(\Omega_1)} \rightarrow \infty$. Since $\mathcal{N}(T)$ does not contain functions vanishing on $\Omega_1$, it follows that $T\chig_{\s \Omega_1}\neq 0$. Then, by Theorem 1.2, the functional $\|Tu_n-v\|^2_\Y + \alpha_2 J_{\s 0}^{\s \Omega_1}(u_n)$ is $BV$-coercive; i.e:
\begin{equation}\label{eq:caca1}
\|Tu_n-v\|^2_\Y + \alpha_2 J_{\s 0}^{\s \Omega_1}(u_n)\rightarrow\infty.
\end{equation}
Now clearly
\begin{equation} \label{eq:caca2}
\begin{aligned}
\|Tu_n-v\|^2_\Y + \alpha_2 J_{\s 0}^{\s \Omega_1}(u_n) &\le \|Tu_n-v\|^2_\Y +\alpha_{\scriptscriptstyle 2}\,\sup_{\vec\nu\in\mathcal V_{\s \theta}}\int_{\Omega_{\s \,0}^{\s \,c}}-u_n\, \text{div}(\theta \vec\nu)\, dx \\
&\le F_{\s \theta}(u_n).
\end{aligned}
\end{equation}
From (\ref{eq:caca1}) and (\ref{eq:caca2}) it follows that $F_{\s \theta}(u_n)\rightarrow \infty$. Hence $F_{\s \theta}$ is $L^{\s 2}(\Omega)$-coercive. The existence of a global minimizer then follows.
%
Finally, the hypothesis that $\mathcal{N}(T)$ does not contain functions vanishing on $\Omega_1$ also implies that $F_{\s \theta}(u)$ is strictly convex and therefore such a global minimizer is unique.
\end{proof}

\section{Signal restoration with $L^2$-$BV$ regularization}
\label{sec:3}

The purpose of this section is to show some applications of the regularization method developed in the previous section consisting in the simultaneous use of penalizers of $L^2$ and of bounded-variation (BV) type to signal restoration problems.

\medskip

A basic mathematical model for signal blurring is given by convolution, as a
Fredholm integral equation of first kind:
\begin{equation}\label{eq:I}
v(t)=\int_0^1 k(t,s) u(t) ds,
\end{equation}
where  $k(t,s)$ is the blurring kernel or point spread function, $u$ is the exact (original) signal and $v$ is the blurred signal. For the examples that follow we took a Gaussian blurring kernel, i.e. $k(t,s)=\frac{1}{\sqrt{2\pi}\sigma_b}\,\textrm{exp}\left(-\frac{(t-s)^2}{2\sigma_b^2}\right)$, with  $\sigma_b> 0$. Equation (\ref{eq:I}) was discretized in the usual way (using collocation and quadrature), resulting in a
discrete model of the form
\begin{equation}\label{eq:II}
    Af=g,
\end{equation}
where $A$ is a $(n+1) \times (n+1)$ matrix, $\,f,g \in \mathbb{R}^{n+1}$ ($f_j=u(t_j), \, g_j=v(t_j),\, t_j=\frac j n, \, 0\leq j\leq n$). We took $n=130$ and $\sigma_b=0.05$. The data $g$ was contaminated with a 1\%
zero-mean Gaussian additive noise (i.e. standard deviation equal to 1\% of the range of $g$).

\medskip

\textbf{Example 3.1.} For this example, the original signal (unknown in real life problems) and the blurred noisy signal which constitutes the data
of the inverse problem for this example are shown in Figure 1.

\begin{figure}[H]
\begin{center}
\includegraphics[scale=.5]{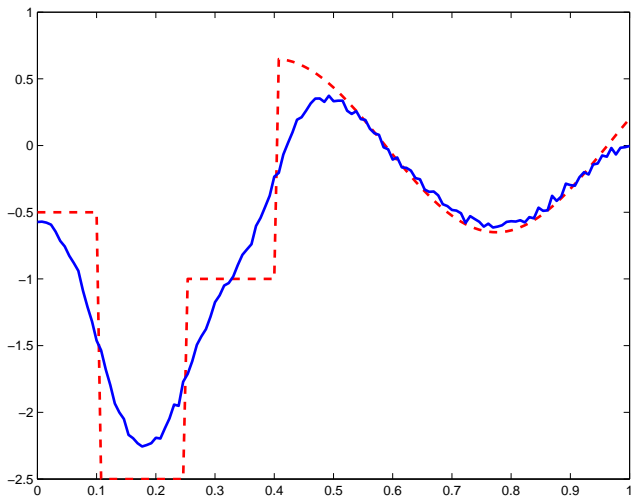}
\caption{Original signal ({\color{red}{- -}}) and blurred noisy signal ({\color{blue}{---}}).}
\label{fig:1}
\end{center}
\end{figure}

Figure 2 shows the regularized solutions obtained with the classical Tikhonov-Phillips method of order zero (left) and with penalizer associated to the bounded variation seminorm $J_{\s 0}$ (right). As expected, the regularized solution obtained with the $J_{\s 0}$ penalizer is significantly better than the one obtained with the classical Tikhonov-Phillips method near jumps and in regions where the exact solution is piecewise constant.
The opposite happens where the exact solution is smooth.

\begin{figure}[ht]
\begin{center}
\includegraphics[scale=.48]{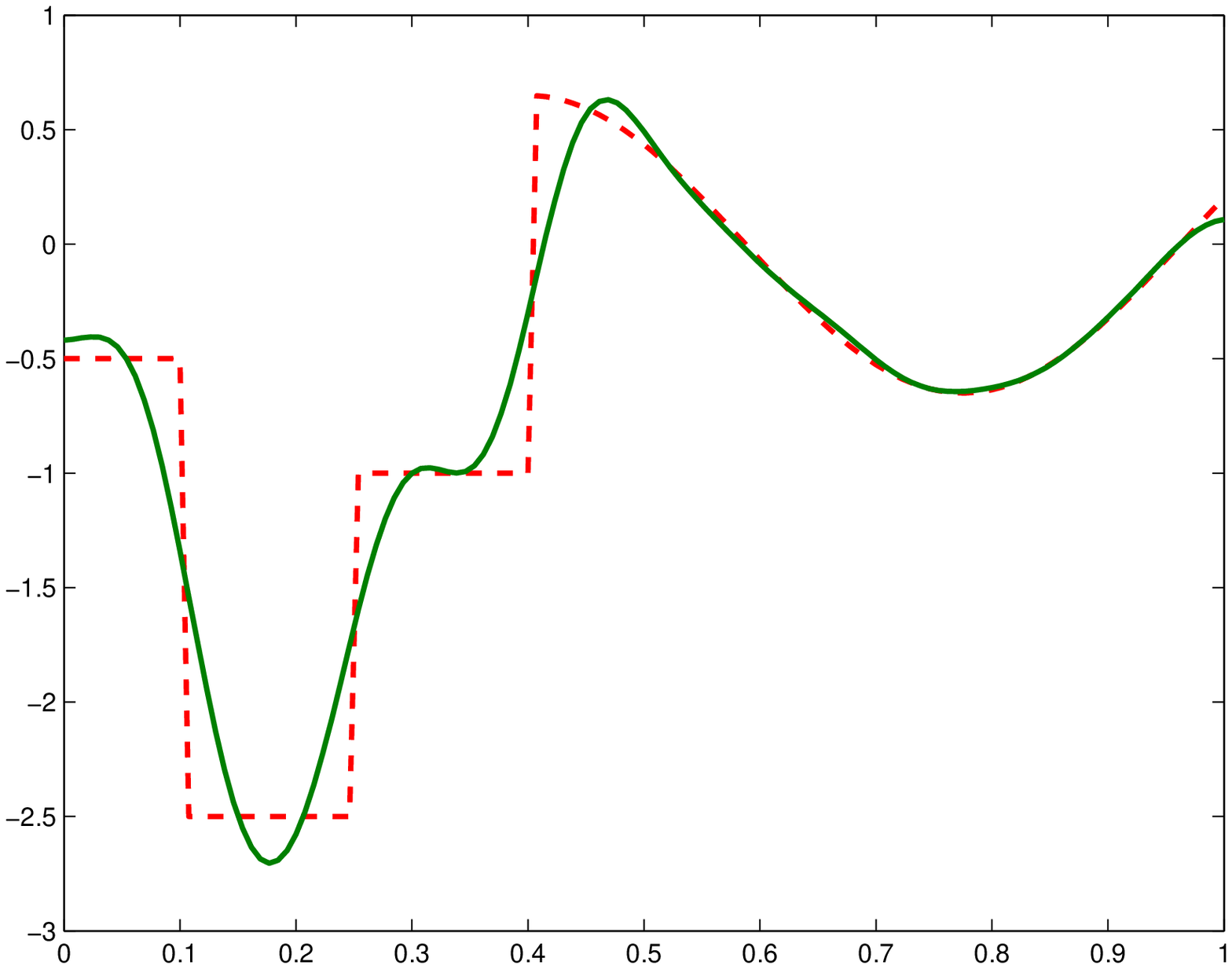}\quad
\includegraphics[scale=.48]{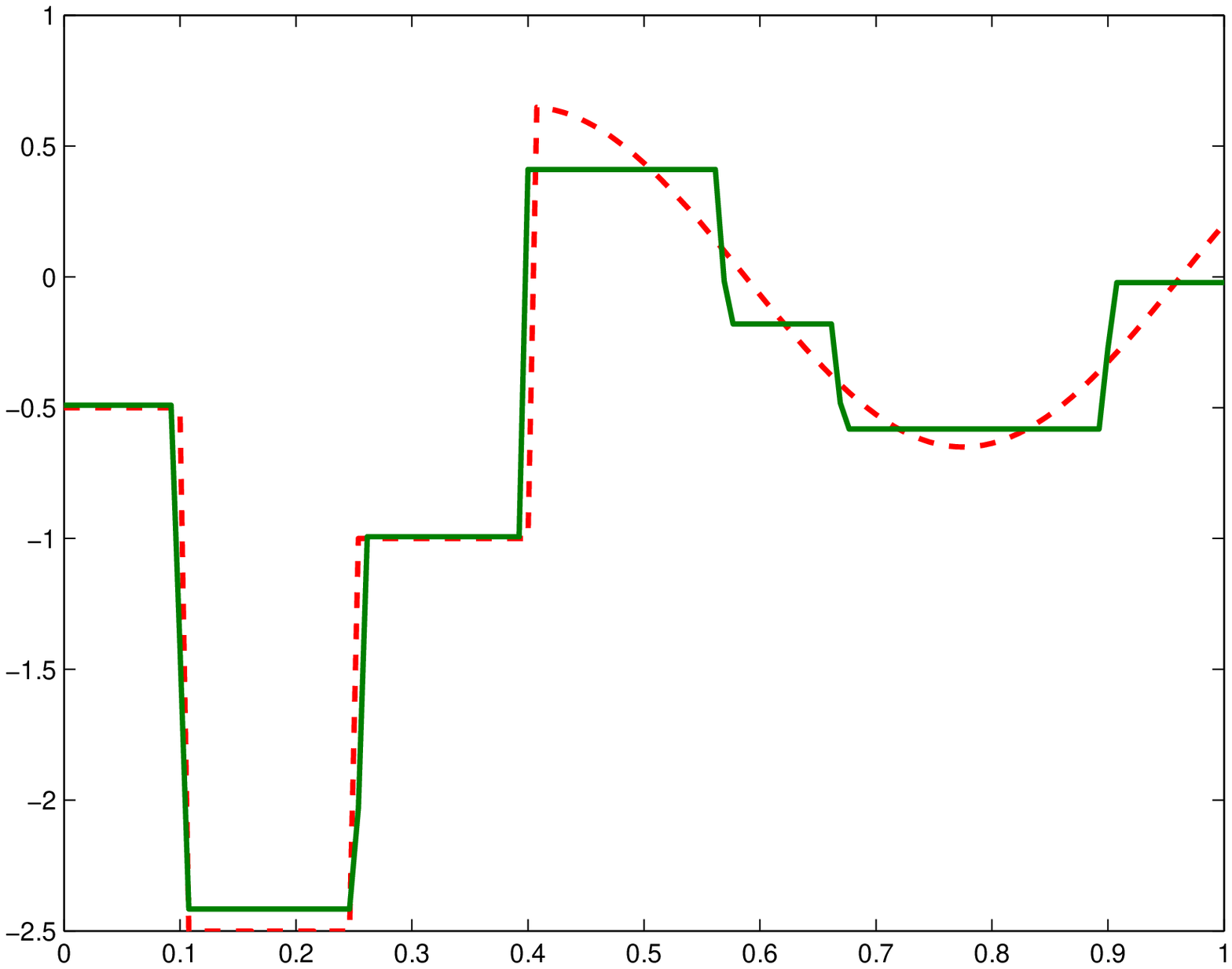}\\
\caption{Original signal ({\color{red}{- -}}) and regularized solutions ({\textcolor[rgb]{0.00,0.50,0.00}{---}}) obtained with Tikhonov-Phillips (left) and bounded variation seminorm (right).}
\label{fig:2}
\end{center}
\end{figure}
Figure 3 shows the regularized solution obtained with the combined $L^2-$BV method (see (\ref{eq:funcional J})). In this case the weight function $\theta(t)$ was chosen to be $\theta(t)\doteq 1$ for $t\in (0, 0.4]$ and $\theta(t)\doteq 0$ for $t\in (0.4, 1)$. Although this
choice of $\theta(t)$ is clearly based upon ``\textit{a-priori}'' information about the regularity of exact
solution, other reasonable choices of $\theta$ can be made by using only data-based information. Choosing a ``good" weighting function $\theta$ is a very important issue but we shall not discuss this matter in this article. For instance, one way of constructing a reasonable function $\theta$ is by computing the normalized (in $[0, 1]$) convolution of a Gaussian function of zero mean and standard deviation $\sigma_b$ and the modulus of the gradient of the regularized solution obtained with a pure zero-order Tikhonov-Phillips method (see Figure 4). For this weight function $\theta$, the corresponding regularized solution obtained with the combined $L^2-$BV method is shown in Figure 5. In all cases reflexive boundary conditions were used (\cite{refb:Hansen2010}) and the regularization parameters were calculated using Morozov's discrepancy principle with $\tau=1.1$ (\cite{refb:Engl-Hanke-96}).
\begin{figure}[H]
\begin{center}
\includegraphics[scale=.461]{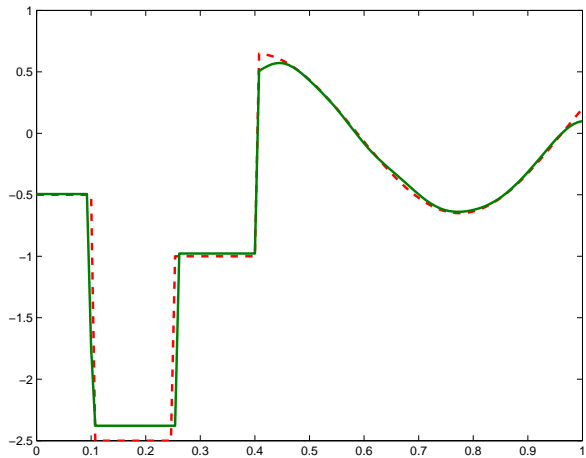}
\caption{Original signal ({\color{red}{- -}}) and regularized solution ({\textcolor[rgb]{0.00,0.50,0.00}{---}}) obtained with the combined $L^2-BV$ method and binary weight function $\theta$.}
\label{fig:3}
\end{center}
\end{figure}
\begin{figure}[H]
\begin{center}
\includegraphics[scale=.5]{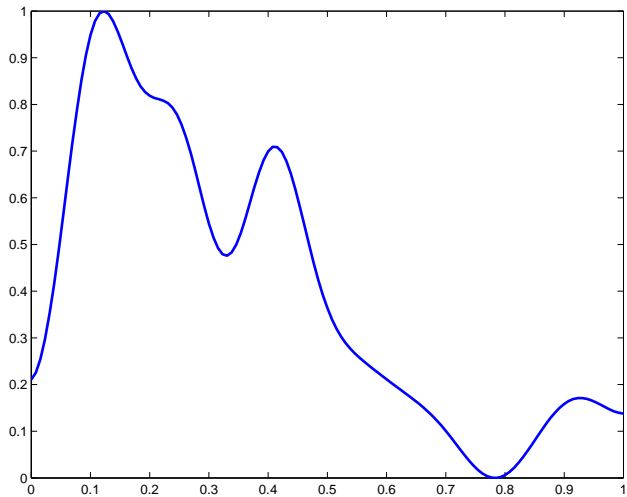}
\caption{Weight function $\theta$ computed by normalizing the convolution of a Gaussian kernel and the modulus of the gradient of the regularized solution with a pure Tikhonov-Phillips method.}
\label{fig:4}
\end{center}
\end{figure}

\begin{figure}[H]
\begin{center}
\includegraphics[scale=.5]{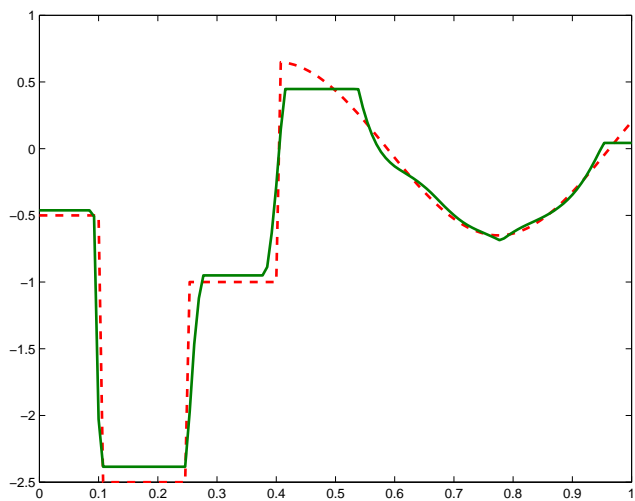}
\caption{Original signal ({\color{red}{- -}}) and regularized solution ({\textcolor[rgb]{0.00,0.50,0.00}{---}}) obtained with the combined $L^2-BV$ method and the data-based weight function $\theta$ showed in Fig. \ref{fig:4}.}
\label{fig:5}
\end{center}
\end{figure}

 As it can be seen, the improvement of the result obtained with the combined $L^2-BV$ method and ``\textit{ad-hoc}'' binary function $\theta$ with respect to the pure simple methods, zero-order Tikhonov-Phillips and pure $BV$, is notorious. As previously mentioned however, in this case the construction of the function $\theta$ is based on ``\textit{a-priori}'' information about the exact solution, which most likely will not be available in concrete real life problems. Nevertheless, the regularized solution obtained with the data-based weight function $\theta$ shown in Figure \ref{fig:4} is also significantly better than those obtained with any of the single-based penalizers. This fact is clearly and objectively reflected by the Improved Signal-to-Noise Ratio (ISNR) defined as
$$ISNR=10 \log_{10}\left(\frac{\norm{f-g}^2}{\norm{f-f_\alpha}^2}
\right),$$ where $f_\alpha$ is the restored signal obtained with regularization parameter $\alpha$. For all the
previously shown restorations, the ISNR was computed in order to have a parameter for objectively measuring and comparing the
quality of the regularized solutions (see Table \ref{tab:1}).
\begin{table}[H]
\caption{ISNR's for Example 3.1.}
\label{tab:1}       
\smallskip
\hskip 2cm
\begin{tabular}{|p{11cm}|p{1.5cm}|}
\hline
\textbf{Regularization Method} & \textbf{ISNR} \\
\hline
Tikhonov-Phillips of order zero & 2.5197\\
\hline
Bounded variation seminorm & 4.2063\\
\hline
Mixed $L^2-$BV method with binary $\theta$ & 5.7086\\
\hline
Mixed $L^2-$BV method with zero-order Tikhonov-based $\theta$  & 4.4029\\
\hline
\end{tabular}
\end{table}

\medskip

\textbf{Example 3.2.} For this example we considered a signal which is smooth in two disjoint intervals and it is piecewise constant in their complement, having three jumps. The signal was blurred and noise was added just as in the previous example. The original and blurred-noisy signal are depicted in Figure 6.
\begin{figure}[H]
\begin{center}
\includegraphics[scale=.5]{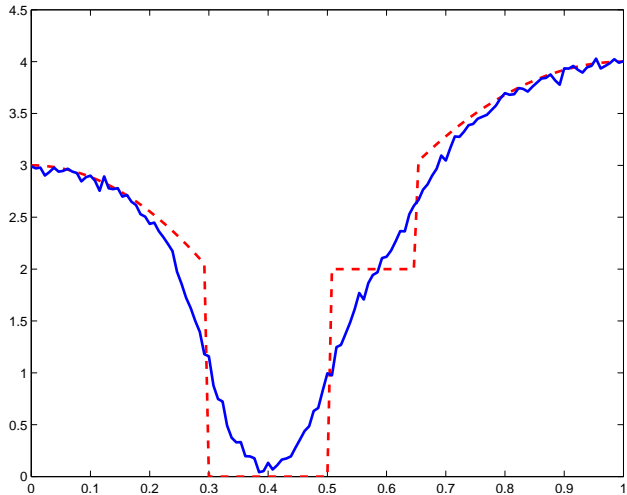}
\caption{Original ({\color{red}{- -}}) and blurred-noisy ({\color{blue}{---}}) signals for Example 3.2.}
\label{fig:6}
\end{center}
\end{figure}
Figure 7 shows the restorations obtained with the classical zero-order Tikhonov-Phillips method (left) and $BV$ with penalizer $J_{\s 0}$ (right).
\begin{figure}[H]
\begin{center}
\includegraphics[scale=.48]{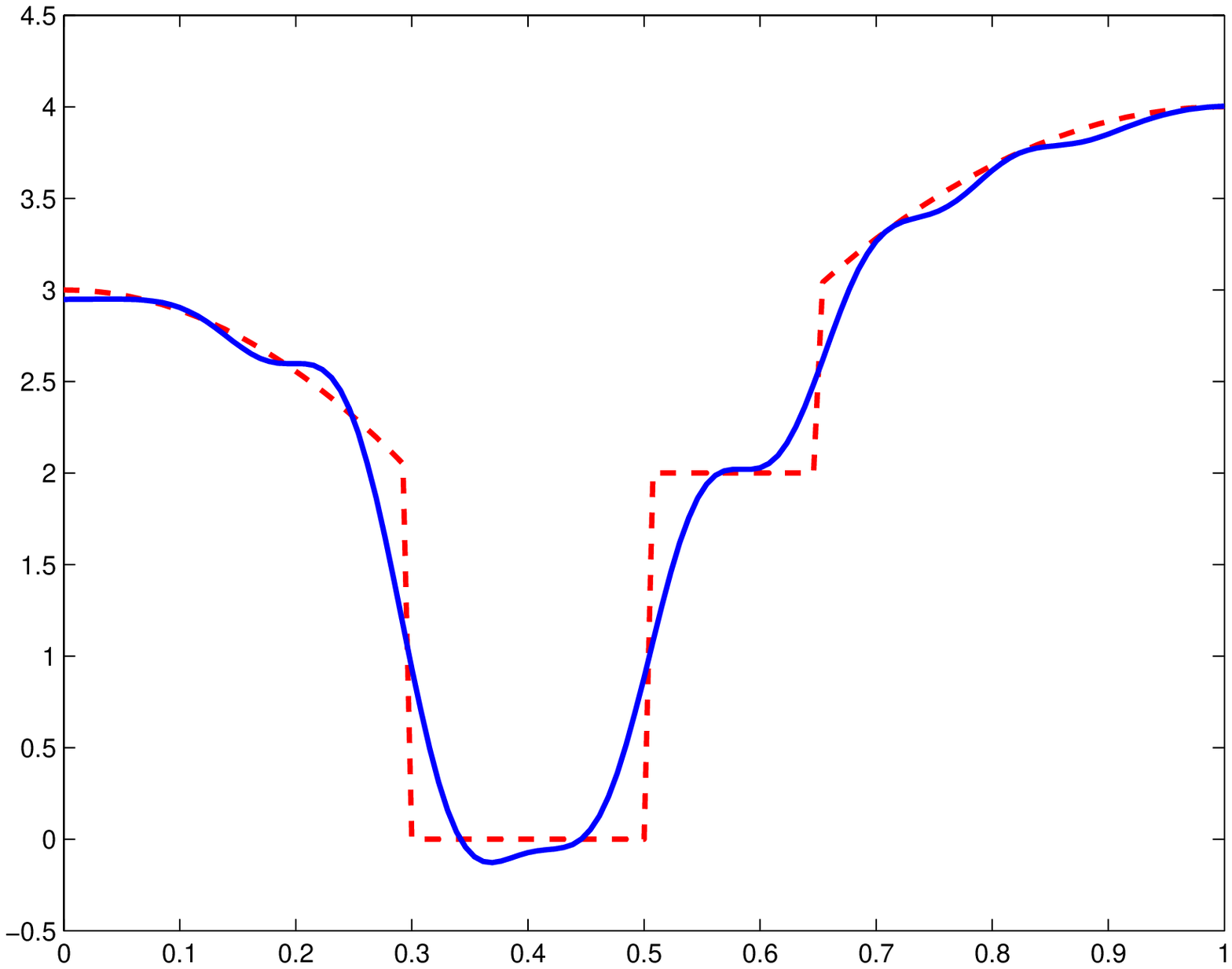}\quad
\includegraphics[scale=.48]{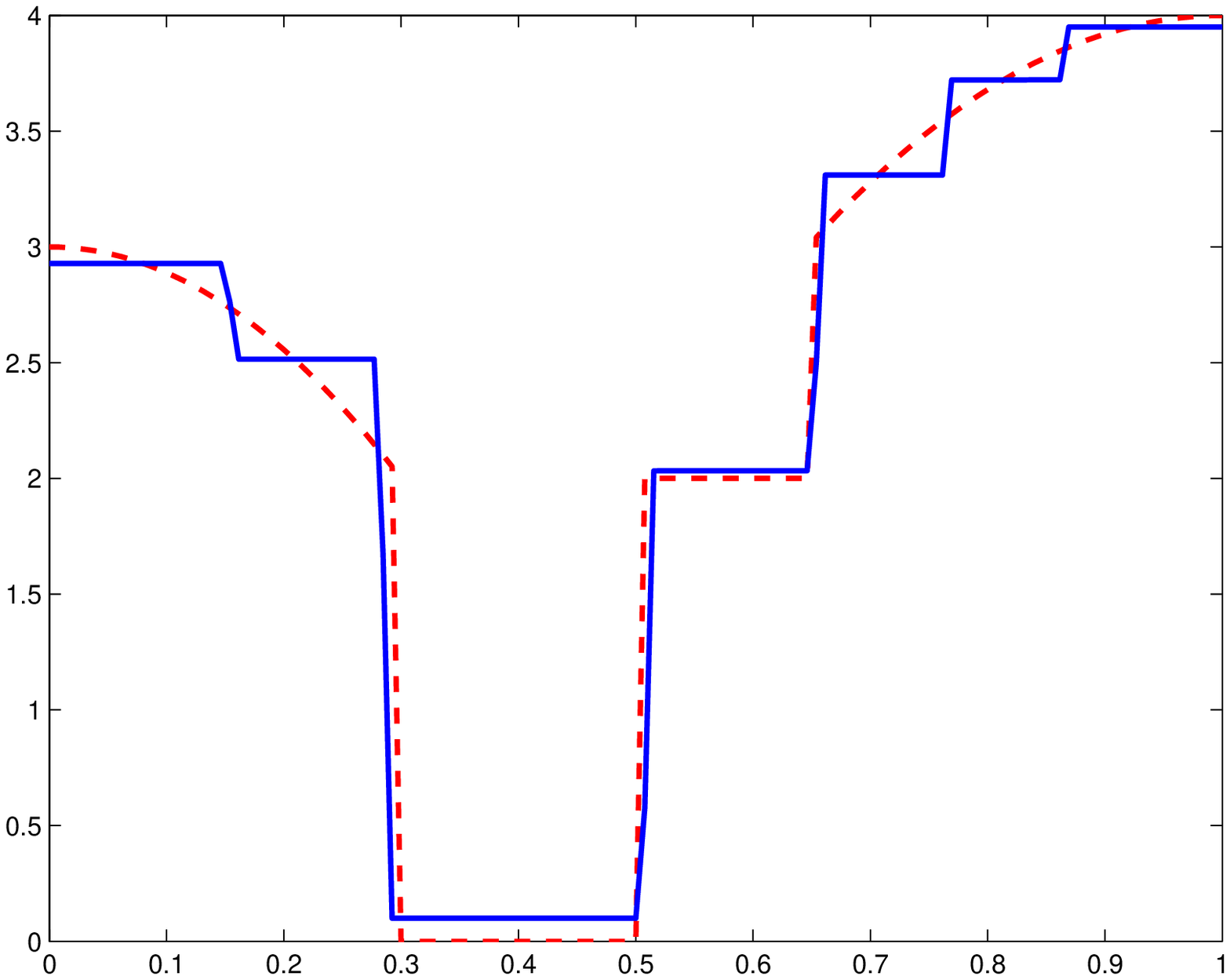}\\
\caption{Original signal ({\color{red}{- -}}) and regularized solutions ({\textcolor[rgb]{0.00,0.50,0.00}{---}}) obtained with Tikhonov-Phillips (top) and bounded variation seminorm (bottom).}
\label{fig:7}
\end{center}
\end{figure}
An ad-hoc binary weight function theta for this example was defined on the interval $[0,1]$ as $\theta(t)=\chi_{\s [0.3, 0.65]}(t)$. The regularized solution obtained with this weight function and the combined $L^2-BV$ method is shown in Figure 8. Once again, the improvement with respect to any of the classical pure methods is clearly notorious.
\begin{figure}[H]
\begin{center}
\includegraphics[scale=.5]{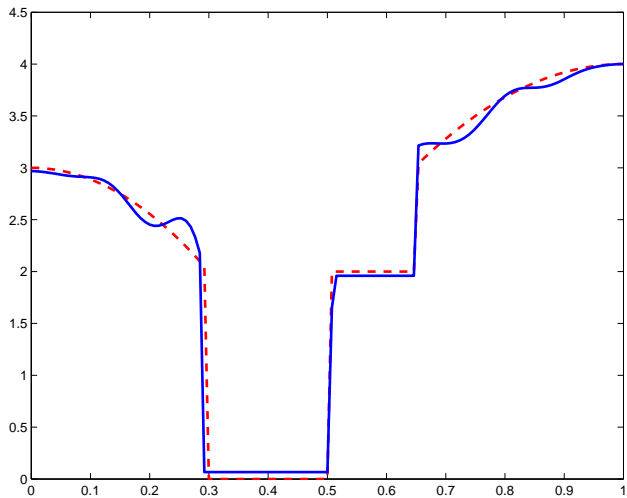}
\caption{Original signal ({\color{red}{- -}}) and regularized solution ({\textcolor[rgb]{0.00,0.50,0.00}{---}}) obtained with the combined $L^2-BV$ method and binary function $\theta$.}
\label{fig:8}
\end{center}
\end{figure}
Here also we constructed a data based weight function $\theta$ as in Example 3.1, by convolving a Gaussian kernel with the modulus of the gradient of a a Tikhonov regularized solution and normalizing the result. This weight function $\theta$ is now depicted in Figure 9, while the corresponding restored signal is shown in Figure 10.
\begin{figure}[H]
\begin{center}
\includegraphics[scale=.5]{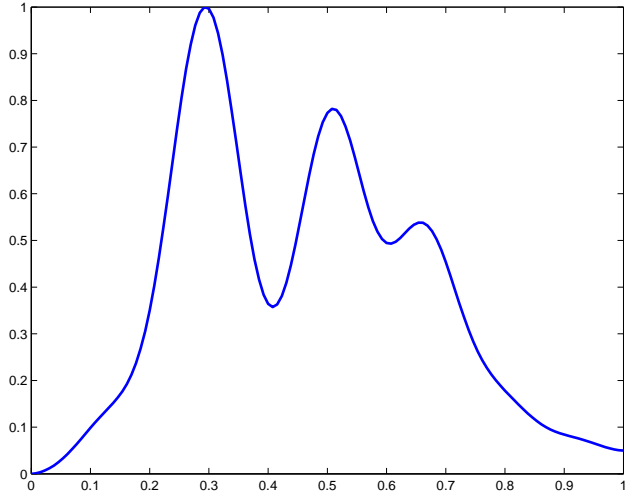}
\caption{Tikhonov-based weight function $\theta$ for Example 3.2.}
\label{fig:9}
\end{center}
\end{figure}
\begin{figure}[H]
\begin{center}
\includegraphics[scale=.5]{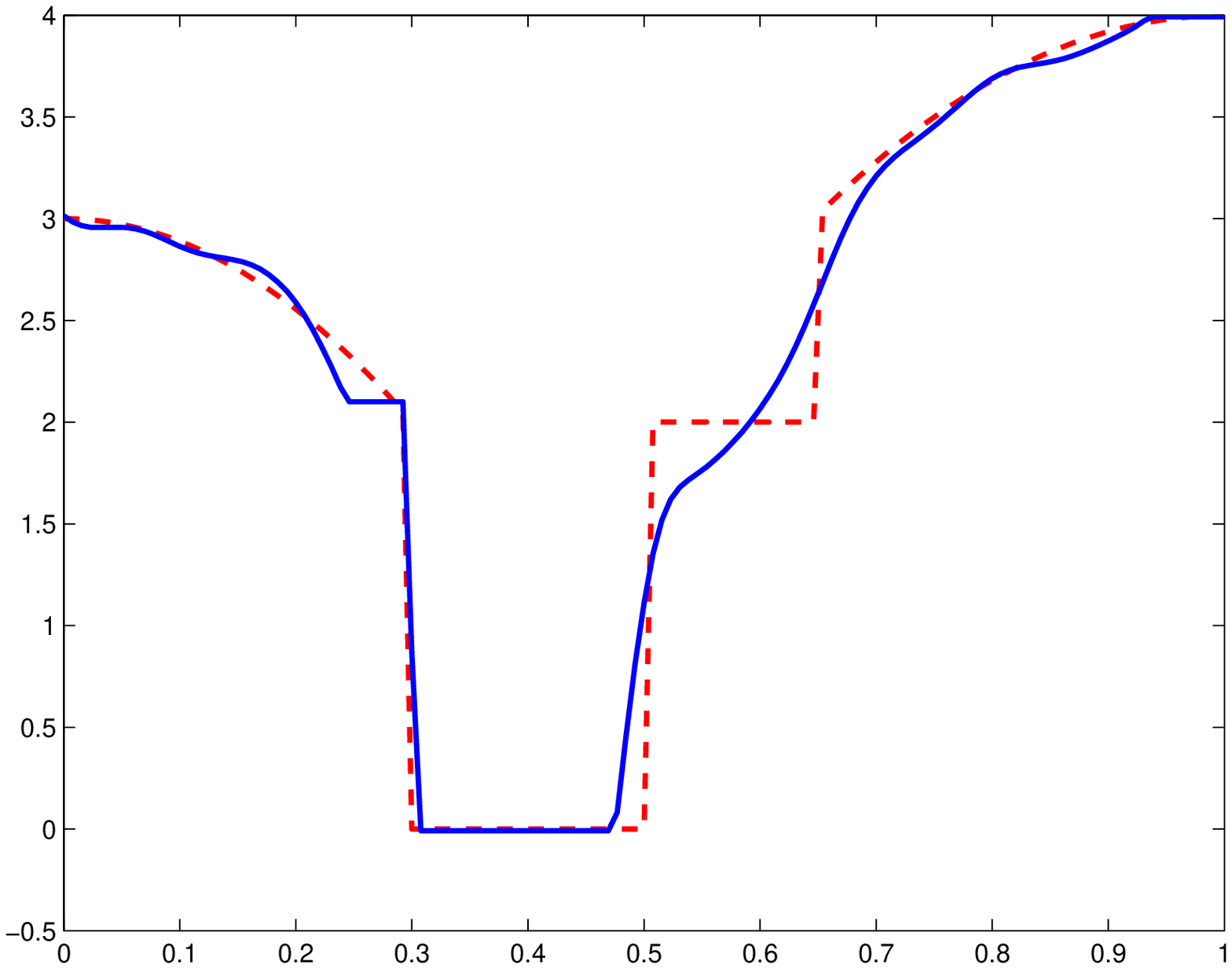}
\caption{Original signal ({\color{red}{- -}}) and regularized solution ({\textcolor[rgb]{0.00,0.50,0.00}{---}}) obtained with the combined $L^2-BV$ method and function $\theta$ showed in Fig. \ref{fig:9}.}
\label{fig:10}
\end{center}
\end{figure}
In table \ref{tab:2} the values of the ISNR for the four restorations are presented. These values show once again a significant improvement of the combined method with respect to any of the pure single methods.
\begin{table}[H]
\caption{ISNR's for Example 3.2.}
\smallskip
\label{tab:2}       
%
%
\hskip 2cm
\begin{tabular}{|p{11cm}|p{1.5cm}|}
\hline
\textbf{Regularization Method} & \textbf{ISNR} \\
\hline
Tikhonov-Phillips of order zero & 2.6008\\
\hline
Bounded variation seminorm & 2.8448\\
\hline
Mixed $L^2-$BV method with binary $\theta$ & 4.8969\\
\hline
Mixed $L^2-$BV method with zero-order Tikhonov-based $\theta$  & 4.3315\\
\hline
\end{tabular}
\end{table}

%
%

\section{Conclusions}

In this article we introduced a new generalized Tikhonov-Phillips regularization method in which the penalizer
is given by a spatially varying combination of the $L^2$ norm and of the bounded variation seminorm. For particular cases, existence and uniqueness of global minimizers of the corresponding functionals were shown. Finally, applications of the new method to signal restoration problem were shown.

Although these preliminary results are clearly quite promising, further research is needed. In particular, the choice or construction
of a weight function $\theta(t)$ in a somewhat optimal way is a matter which undoubtedly deserves much further attention and study.
Research in these directions is currently under way.


\section*{Acknowledgments}
This work was supported in part by Consejo Nacional de Investigaciones Cient\'{\i}ficas y T\'{e}cnicas, CONICET, through PIP
2010-2012 Nro. 0219, by Agencia Nacional de Promoci\'{o}n Cient\'{\i}fica y Tecnol\'{o}gica, ANPCyT, through project PICT 2008-1301, by Universidad Nacional del Litoral, through projects CAI+D 2009-PI-62-315, CAI+D PJov 2011 Nro. 50020110100055, CAI+D PI 2011 Nro. 50120110100294 and by the Air Force Office of Scientific Research, AFOSR, through Grant FA9550-10-1-0018.


\bibliographystyle{amsplain}
\bibliography{ref1}

\end{document}